\documentclass[1p]{elsarticle}

\usepackage{lineno,hyperref}
\modulolinenumbers[5]

\usepackage{url}
\usepackage{amssymb}
\usepackage{mathtools}
\usepackage{caption}
\usepackage{wrapfig}
\usepackage{subcaption} 
\usepackage{tabulary}
\usepackage{tabularx} 
\usepackage{booktabs}
\usepackage{enumitem}

\usepackage{algorithm,algorithmic}
\usepackage{eqparbox,array}

\newcommand{\smallAlign}{\hspace{1em}}

\newcommand{\N}{\mathbb{N}}
\newcommand{\Rplus}{\mathbb{R}_{\geq 0}}

\newdefinition{definition}{Definition}
\newtheorem{theorem}{Theorem}
\newproof{proof}{Proof}

\journal{}

\begin{document}

\begin{frontmatter}

\title{The Dial-a-Ride Problem in Primary Care with Flexible Scheduling}

\author[1]{Christina B\"using\corref{mycorrespondingauthor}}
\ead{buesing@math2.rwth-aachen.de}
\author[1]{Martin Comis}
\ead{comis@math2.rwth-aachen.de}
\author[1]{Felix Rauh\corref{mycorrespondingauthor}}
\cortext[mycorrespondingauthor]{Corresponding author}
\ead{felix.rauh@rwth-aachen.de}

\address[1]{Lehrstuhl II f\"ur Mathematik, RWTH Aachen University, Pontdriesch 10--12, 52062 Aachen, Germany}

\begin{abstract}
Patient transportation systems are instrumental in lowering access barriers in primary care by taking patients to their GPs.
As part of this setting, each transportation request of a chronic or walk-in patient consists of an outbound trip to the GP and an inbound trip back home.
The economic sustainability of patient transportation systems mainly depends on their utilization and how well transportation requests can be bundled through ride sharing.
To ease the latter, we consider a flexible scheduling of chronic patients in which only a certain range for an appointment is fixed a priori while the exact time is determined by the scheduling of the outbound trip.
This leads to a novel extension of the dial-a-ride problem that we call the dial-a-ride problem with combined requests and flexible scheduling (DARPCF).
In this paper, we introduce two heuristics for the DARPCF that exploit this increased flexibility.
Both approaches initially compute so-called mini-clusters of outbound requests.
Then, the mini-clusters are linked by (i) solving a traveling salesman problem and creating routes of outbound rides with a splitting procedure or by (ii) using a rolling horizon approach and solving bipartite matching problems for the vehicle assignment.
Our computational study shows that by using the presented algorithms with the flexible scheduling of chronic appointments, the average number of served requests can be increased by $ 16\% $ compared to a non-flexible setting.
\end{abstract}


\end{frontmatter}


\section{Introduction}
\label{intro}

The aging population in rural areas is facing increasing access barriers when seeking primary care services~\cite{syed2013traveling}.
On the one hand, public transportation systems are often poorly developed and impractical to visit a general practitioner (GP)~\cite{berg2019importance}.
On the other hand, the individual mobility decreases with age and the use of a cab is generally very expensive~\cite{AHERN201227}.
To provide a viable alternative that is both efficient and convenient, we investigate so-called \textit{dial-a-ride systems} for the transportation of patients to GPs.
Dial-a-ride systems offer the comfort of transporting patients directly between their homes and GPs, while being more economical than cab rides by pooling transportations requests.
The standard process of arranging transportation in such systems is the following.
First, a patient contacts their GP to arrange a fixed-time appointment, e.g., Monday 10 a.m.
This appointment is then communicated to the transportation company which in turn makes a transport commitment to take the patient from their home to the appointment and back.
While this procedure is relatively comfortable for GPs and patients, it comes with the downside of not synchronizing appointment scheduling and patient transportation.
As a result, simultaneous transportation requests may be widely spread which prevents an efficient pooling and ultimately implies high transportation cost.

To alleviate this drawback, we propose a new concept for dial-a-ride systems that enables a partial synchronization of appointment scheduling and patient transportation.
We thereby focus on appointments that are known several weeks in advance and introduce the so-called \emph{flexible scheduling}.
The idea of flexible scheduling is to change the arrangement of appointments and transportation as follows.
When a patient arranges an appointment with their GP, only a certain range for the appointment is fixed, e.g., Monday morning.
The transportation company can now flexibly schedule and reschedule the patient's rides within the previously agreed range.
Finally, the transportation company fixes the vehicle routes a few days ahead of transportation and thereby determines the exact appointments which are then communicated to patients and GPs.

While flexible scheduling clearly offers the potential to reduce transportation cost, it requires a certain degree of spontaneity from patients and GPs.
From the patients' point of view, this means that the exact time of an appointment is not known until a few days before it takes place.
We therefore propose flexible scheduling primarily for elderly and chronic patients who can adjust their daily routines and to which we subsequently refer to as \textit{chronic patients}.
From the GPs' point of view, we require the reservation of flexible appointment slots that will only be filled a few days in advance.
As a result, physicians always know if there remain gaps in their schedules, however can only start filling them with fixed-time appointments or walk-in patients shortly beforehand.
Such provider prescribed restrictions on how available slots may be filled are a common concept in appointment scheduling; compare e.g.~\cite{doi:10.1080/07408170802165880}.

From here on, we jointly refer to patients that arrange short-term fixed-time appointments as \textit{walk-in patients}.
As part of our proposed concept, we allow walk-in patients to request transportation for the specific times of their appointments.
However, the transportation company may turn these requests down if they do not fit into the current vehicle routes.
Once a transportation request has been accepted, it must be serviced entirely, i.e.\ neither the \textit{outbound} nor the \textit{inbound} journey may be canceled.

In this paper, we investigate the extension of the classical dial-a-ride problem (DARP) with two combined rides per request (inbound and outbound) by flexible scheduling.
To that end, we introduce the \textit{dial-a-ride problem with combined requests and flexible scheduling} (DARPCF) that allows a flexible scheduling of the outbound rides.
Each flexibly scheduled outbound ride entails a subsequent inbound ride with a fixed time window.
As our main contribution, we present two heuristics for the DARPCF that enable transportation companies to use flexible scheduling.
We call these heuristics the \textit{Mini-Cluster Linking and Insertion Heuristic} (MCLIH) and the \textit{Mini-Cluster Matching Algorithm} (MCMA) and they both consist of two phases:
First, the requests of chronic patients in the DARPCF setting are processed
and second, the requests of walk-in patients are inserted by an online algorithm.
Both heuristics start by clustering the chronic patient requests.
MCLIH then solves a traveling salesman problem and creates routes of outbound rides with a splitting procedure and greedily inserts inbound rides.
MCMA uses a rolling horizon approach and solves bipartite matching problems for the vehicle assignment.
Finally, we compare the performances of MCLIH and MCMA on realistic test instances.
Because of the large number of requests in the considered scenarios, the solution quality is not evaluated by the length of the vehicle routes but rather the number of served customers.
Thereby, we illustrate how flexible scheduling can accommodate up to $16\%$ more patient requests compared to traditional systems.

The remainder of this paper is structured as follows:
We start by reviewing relevant literature in Section \ref{sec:1}.
Then, in Section \ref{sec:DARPform}, we formally introduce the DARPCF, followed by a description of the solution approaches in Section \ref{sec:solutions}.
The results of our computational study are presented in Section \ref{sec:CompStudy}.
Finally, Section~\ref{sec:conclusion} concludes with a short summary and outlook on future work.

\section{Literature Review}
\label{sec:1}
Dial-a-ride problems belong to the most classical optimization problems in transportation.
For a detailed review on models techniques we refer to the extensive surveys of Molenbruch et al.~\cite{Molenbruch2017} and Ho et al.~\cite{Ho2018}.
Besides an overview of the most relevant historical publications for the DARP, we focus on recent developments and problem settings similar to one studied in this paper.

\paragraph{Early Research}
The study of dial-a-ride problems started with work of Wilson et al.~\cite{Wilson1971} who examined solutions to dial-a-ride systems in North American cities (Haddonfield, NJ and Rochester NY).
This work dates back to 1971 and was improved in a much-cited publication by Jaw et al.~\cite{Jaw1986}.
The authors consider time windows on departure or arrival and use an insertion heuristic which builds up schedules for the vehicles by the successive insertion of transportation requests.
The objective function is non-linear and takes various service related constraints into account.
A further improvement is due to Madsen et al.~\cite{Madsen1995}, who proposed the so-called REBUS algorithm in 1995.
It is a fast insertion heuristic which can also process requests interactively, thus solving the \textit{dynamic} DARP.
An essential concept used is that of the \textit{time slack} of a stop in the schedule, which is the largest increase of the departure time that is possible without violating any time window constraints.
REBUS can be characterized as a greedy algorithm because the schedule is built successively by inserting transportation requests in a fashion such that it minimizes the cost function.

The first exact approach for solving the single-vehicle DARP was published in 1980 by Psaraftis \cite{Psaraftis1980} who solved the problem with up to nine users through dynamic programming. However, the algorithm does not consider time windows but only so-called \textit{maximum position shifts}.

Bodin and Sexton \cite{Bodin1986} did a first remarkable effort on cluster-first route-second methods in 1986.
After using a clustering method for assigning requests to vehicles, a heuristic single-vehicle solution is computed, which involves a  Benders' decomposition and a so-called \textit{space-time heuristic}.

The notion of mini-clusters which we will present in this work was introduced by Desrosiers et al.~\cite{Desrosiers1988}.
The authors propose an algorithm where mini-clusters are generated by using neighboring criteria and the routing problem is solved by column generation.
Ioachim et al.~\cite{Ioachim1995} improved this approach by applying a column generation algorithm in the mini-clustering phase.
They presented results for instances with up to $3000$ users.

\paragraph{Recent Developments}
In the $21$st century, research on the DARP advanced all over the world using various types of techniques.
Noteworthy work was done at Universit\'e de Montr\'eal and the associated HEC Montr\'eal Business School by Cordeau and Laporte~\cite{Cordeau2003}.
They introduced a new heuristic applying Tabu Search and their problem definition has become the standard DARP setting~\cite{Molenbruch2017}.
We present this problem definition in Section \ref{sec:DARPform}.
Moreover, Cordeau~\cite{Cordeau2006} published an MILP formulation widely used today and introduced new valid inequalities used in Branch-and-Cut algorithms.
The artificial instances used by Cordeau were extendend by R\o{}pke et al.~\cite{Ropke2007} and have become important benchmark instances for other authors~\cite{Molenbruch2017}.

According to a recent and extensive survey by Molenbruch et al.~\cite{Molenbruch2017}, the most efficient algorithm known at that time (having solved all of Cordeau's instances with up to $96$ users) was a Branch-and-Cut-and-Price method by Gschwind and Irnich~\cite{Gschwind2015}.
The subsequent work by Gschwind and Drexl~\cite{Gschwind2019} is an adaptive large neighborhood search metaheuristic and is claimed to be competitive with all state-of-the-art heuristics.

\paragraph{Flexible Scheduling and Return Trips}
To the best of our knowledge, the flexible scheduling concept which we introduce in this article has not yet been considered.

A related field in health care where scheduling and routing are combined is home health care (HHC) as presented in a review by Fikar and Hirsch~\cite{Fikar2017}.
The approaches in that field also aim at combining different decision levels in order to improve operations.
In the regular setting, this comprises organizing shifts, the assignments of nurses to patients, and routing decisions.
However, the resulting routing problem does not compute a route with pick-up and delivery of patients as in the DARP, but rather minimizes the operator's routing cost or other objective functions without any kind of ride sharing or relevant vehicle capacity constraints (e.g.~\cite{Cappanera2015,Grenouilleau2019}).

More closely related to vehicle routing but without joint scheduling, a recent paper by Adelh\"utte et al.~\cite{Adelhuette2021} considers patient transportation with incomplete information and semi-plannable transports.
This means, that e.g.~for dialysis patients the outbound trip to the treatment is given with complete information, whereas the return trip has initially unknown time windows.
The authors formulate a Vehicle Routing Problem with General Time Windows and show in their numerical study that incorporating semi-plannable transport clearly reduces waiting times compared to the previous scheduling method.

A similar result can be found in a paper by Schilde et al.~\cite{Schilde2011} who consider a DARP setting where stochastic information if an outbound request causes a corresponding inbound request is taken into account.
The problem arises in the operation of the Austrian Red Cross and the presented results show that by using the stochastic information an improvement of around $ 15\% $ in the lexicographic objective function (with primary objective total tardiness) can be achieved.
The used algorithms are modifications of the two metaheuristic approaches of variable neighborhood search (VNS) and multiple plan approach (MPA)~\cite{Mladenovic1997,Bent2004}.
In the first, the stochastic information about the return trips is used for comparing candidate solutions by generating various sets of sample return trips and inserting them into the candidate solutions.
The candidates with better average results are eventually preferred.
The MPA extension, on the other hand, uses sample return trips for extending initial solutions and eventually removes the return trips again in order to produce gaps for the insertion of the actual return trips.

In the next section, we describe our problem setting and formulation which introduces the set of \textit{chronic patients} and a more flexible way to handle their appointments and associated transportation requests.

\section{Problem Description and Notation}
\label{sec:DARPform}

We begin this section by introducing the standard problem setting for the static (i.e.~offline) dial-a-ride problem as presented by Cordeau and Laporte \cite{Cordeau2003}, which was adopted by most authors in recent years \cite{Molenbruch2017}.
Then, we discuss the DARPCF and its alterations to the DARP through the concept of flexible scheduling.

\subsection{The Static Dial-a-Ride Problem}
\label{sec:staticDARP}
Let $n\in \N$  denote the number of requests (or patients) to be served.
We define the DARP on a complete directed graph $G=(N,A)$, called the \textit{DARP road graph}, where $ N = P \cup D \cup \{ 0, 2n+1 \},$ $ P = \{1,2,\dots,n \}$ and $ D = \{n+1,n+2,\dots,2n \} $.
By $ R = \{r_1, \dots, r_n \} $, we denote the set of all requests where
each request $ r_i = (i,n+i) \in R \subseteq P \times D $ is represented by a pick-up node $ i \in P $ and a drop-off (delivery) node $ n+i \in D $.
The nodes $ 0 $ and $ 2n+1 $ represent an origin and a destination depot for the fleet of $ m \in \N $ vehicles.
We denote by $ t_{jk} \in \mathbb{R}_{\geq 0} $ the non-negative travel time between nodes $ j\in N $ and $ k \in N $ and we assume that the triangle inequality $ t_{jk} \leq t_{jl}+t_{lk} $ holds for all $j,k,l \in N$.
We associate a load $ q_k \in \mathbb{Z} $ with each node $ k \in N $, by default $ +1 $ for pick-up nodes and $ -1 $ for delivery nodes.
The time window for a node $ k \in N $ is denoted by $ [e_k, l_k] $ for $ e_k, l_k \in \mathbb{R},  e_k \leq l_k $ and has a maximum length of $ W \in \mathbb{R}_{\geq 0} $, i.e.~$l_k-e_k \leq W$.

Moreover, let $ Q \in \mathbb{N} $ denote the passenger capacity of each vehicle and $ L \in \mathbb{R}_+ $ the maximum route length for each vehicle.
The maximum user ride time $ L_i \in \mathbb{R}_+ $ of request $ r_i $ for $ i=1,\dots,n $ describes the maximum time the user may be on the vehicle for the ride.
It can either be constant for all $ i \in {1,\dots,n} $ or proportional to the direct ride time, e.g.~$ L_i = 1.5 \cdot t_{i,n+i} $.

As done by Cordeau and Laporte \cite{Cordeau2003}, we assume that for outbound trips only a time window for the delivery is specified, while for inbound trips there is only a time window for the pick-up.
All other time windows are not specified explicitly, but derived implicitly from the maximum and minimum ride time constraints.
For example, consider a customer requesting an outbound ride with a direct ride time of $30$ minutes and a maximum ride time of $45$ minutes.
If the time window for the delivery is set to $ [10{:}45,\,11{:}00] $, the implicit time window for the pick-up would be $ [10{:}00,\,10{:}30] $ because no other pick-up time could result in a feasible ride; compare Figure~\ref{fig:timeWindows}.

\begin{figure}
	\centering
			\includegraphics{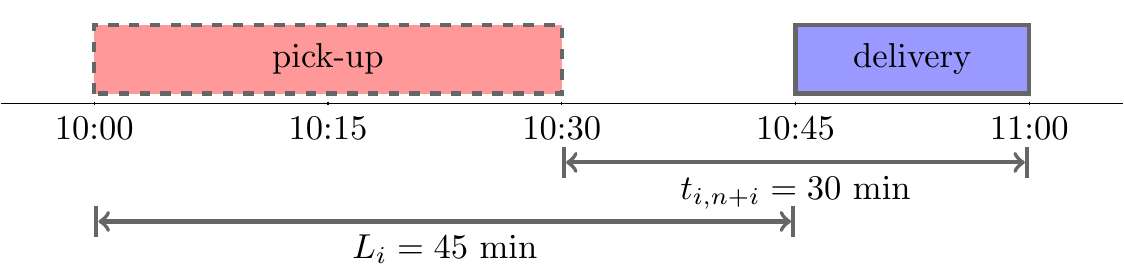}
	\caption{Implicit pick-up time window for outbound ride $r_i=(i,n+i)\in R$ with direct travel time $t_{i,n+i} \in \Rplus$ and maximum user ride time $L_i\in \Rplus$.}
		\label{fig:timeWindows}
\end{figure}

An optimal solution consists of a set of $m$ vehicle routes from node $ 0 $ to node $ 2n+1 $ such that for each request $ r_i $ for $ i=1,\dots,n $ the nodes $ i $ and $ n+i $ are contained in the same vehicle route in the correct order (the so-called \textit{precedence constraint}).
Moreover, the capacity and service constraints need to be satisfied and the routing cost must be minimized.

\subsection{DARP with Combined Requests and Flexible Time Windows (DARPCF)}\label{sec:DARPCF}
We extend the above problem setting in order to describe the problem of flexible scheduling of chronic patients.
User requests now consist of two rides:
an \textit{outbound trip} to the appointment without time window (or more precisely a time window which comprises the whole service period), and an \textit{inbound trip} which has to take place within a time window of length $ W $ that starts when the stay at the GP of duration $ d_{\text{GP}}\in \Rplus $ ends.
As a simplification, we assume that $d_{\text{GP}} \in \Rplus$ is constant for all patients.

The adaptation of the discussed model is straightforward:
The set $ R_c \subseteq P \times D $ only includes the chronic outbound requests.
For the corresponding inbound trips, we introduce new vertex sets $ \overline{D} $ and $ \overline{P} $ which are copies of $ D $ and $ P $, respectively.
The inbound trip belonging to outbound trip $ r_i = (i, n+i) $ is now given by $ \overline{r_i} = (\overline{n+i}, \overline{i}) \in \overline{D}\times\overline{P} $.
Thus, the set $ \overline{R_c} := \{\overline{r_1}, \dots, \overline{r_n}\} \in \overline{D}\times\overline{P}  $ denotes the set of chronic inbound requests.
The travel times for nodes $ j, k $ and their copies $ \overline{j}, \overline{k} $ are set canonically to
\[ t_{j\overline{j}} := 0,\smallAlign t_{\overline{jk}} := t_{\overline{j}k} := t_{j\overline{k}}:=t_{jk}. \]
The load of a copy is the negative of the original vertex, i.e.~$ q_{\overline{j}} = -q_j, j\in P \cup D $,
and the time window for request $ \overline{r_i} $ is denoted by $ [ e_{\overline{i}}, l_{\overline{i}} ] $.

\begin{definition}[DARPCF]
	Let the set of requests $ R $ consist of $ n $ pairwise outbound and inbound requests, i.e. $ R=R_c \cup \overline{R_c} \subseteq P \times D \cup \overline{D}\times\overline{P} $ and let the DARP road graph $ G=(N,A) $ with  $ N = P \cup D \cup \overline{D}\cup\overline{P} \cup \{ 0, 2n+1 \} $ be the complete directed graph which contains a pick-up and a delivery node for each request.
	There is a homogeneous fleet of $ m $ vehicles with start depot at node $ 0 $ and end depot at node $ 2n+1 $.
	Furthermore, $ d_{\text{GP}} $ denotes the duration of stay at the GP and $ W $ is the time window length.
		
	We define the \textit{DARP with Combined Requests and Flexible Time Windows (DARPCF)} as the problem of finding $ m $ vehicle routes on $ G $ which serve the pairwise requests in $ R $ in the following fashion:
	If an outbound request $ r_i $ is scheduled to arrive at time $ t_i $, then the departure time $ \overline{t_i} $ of the corresponding inbound request $ \overline{r_i} $ must satisfy $ \overline{t_i}\in [t_i+d_{\text{GP}},t_i+d_{\text{GP}}+W] $.
	Moreover, the vehicle capacities, maximum route length and the maximum user ride time must be respected.	
\end{definition}

This extension makes the problem more complex since time windows of corresponding outbound and inbound trips are linked:
We do not allow long waiting times between the outbound and inbound ride of a patient, or even worse, to schedule the inbound trip before the outbound trip.
The simplest possible remedy one could think of is to require that both trips are served by the same vehicle.
However, this is not a promising approach as it could easily lead to vehicles having to wait inactively during the entire appointment after the delivery of a patient.

In the definition above, we only model the requests of chronic patients.
In a next step, we extend this definition to (walk-in) patients with online requests.
In that context, we introduce the overall objective of the problem.

\subsection{Extended DARPCF with Online Requests}
In order to better reflect the reality of primary care, we also consider so-called walk-in patients.
Their requests are also pairwise outbound and inbound requests, but they are not known at the time of processing the chronic requests.
Precisely, apart from time windows and the other properties from the regular DARP setting, each pair $ (r^w_i,\overline{r^w_i}) $ of outbound and inbound walk-in requests is characterized by a release time $ s_i $ which symbolizes the time when the patient calls the operator and requests transportation.
We denote the sets of walk-in outbound and inbound requests by $ R_w $ and $ \overline{R_w} $, respectively.
\begin{definition}[Extended DARPCF]
	Let  $ R=R_c \cup \overline{R_c} \cup R_w \cup \overline{R_w} $ be a set of requests consisting of pairwise requests of both chronic and walk-in patients.
	We define the \textit{Extended DARPCF with Online Requests} as the problem aiming to solve the following two problems:
	\begin{enumerate}[label=(\roman*)]
		\item The DARPCF for requests $ R_c \cup \overline{R_c} $, thus creating vehicle schedules $ \mathcal{S} $.
		\item Online DARP on $ \mathcal{S} $ for requests $ R_w \cup \overline{R_w} $, i.e.~insert each pair of requests $ (r^w_i,\overline{r^w_i}) \in  R_w \times \overline{R_w} $ into $ \mathcal{S} $ such that schedules before release time $ s_i $ remain unchanged.
	\end{enumerate}
\end{definition}
Note that we do not assume that a feasible solution serving all walk-in requests exists, but rather aim to maximize the total number of served requests.

\section{Solution Methodology}
\label{sec:solutions}

In the following, we introduce two solution procedures for the Extended DARPCF, namely the  \textit{Mini-Cluster Linking and Insertion Heuristic} (MCLIH) and \textit{Mini-Cluster Matching Algorithm} (MCMA).
Both algorithms consist of two phases:
In the first phase, all chronic patient requests are scheduled and in the second phase, all online requests are added to the schedule. 

To obtain a schedule for the chronic patient requests, both algorithms use a greedy mini-clustering algorithm which we present in Section \ref{sec:miniclustering}.
The algorithm calculates a partition $ \mathcal{V} = \{ \mathcal{M}_1, \dots, \mathcal{M}_{|\mathcal{V}|} \} $ of the set $ R_c $ of all chronic outbound trips.
Subsequently,  an optimal route within each of these so-called \textit{mini-clusters} is computed.
Finally, the mini-clusters need to be linked to obtain the daily route for every vehicle. 

This linking procedure differs for MCLIH and MCMA.
MCLIH solves a capacitated vehicle routing problem (CVRP) on the mini-clusters $\mathcal{V}$.
To that end, it links the mini-clusters to a single traveling salesman tour that is then split into the respective vehicle routes.
Subsequently, a \textit{greedy insertion heuristic} (GIH) based on the ones in \cite{Jaw1986,Madsen1995} is used to insert the chronic inbound trips. 
This GIH explores all possible insertions into the existing schedules and greedily chooses the one that minimizes a prespecified metric.

MCMA, on the other hand, directly includes the chronic inbound trips when linking the mini-clusters.
To that end, it uses a rolling time horizon approach and repeatedly solves specialized matching problems to assign vehicles to requests.

Eventually, both MCLIH and  MCMA solve the online extension of the DARPCF by using the GIH for the insertion of walk-in patients.
An overview of the different subroutines in MCLIH and MCMA can be found in Figure~\ref{fig:overview}.

\begin{figure}
	\centering
	\includegraphics{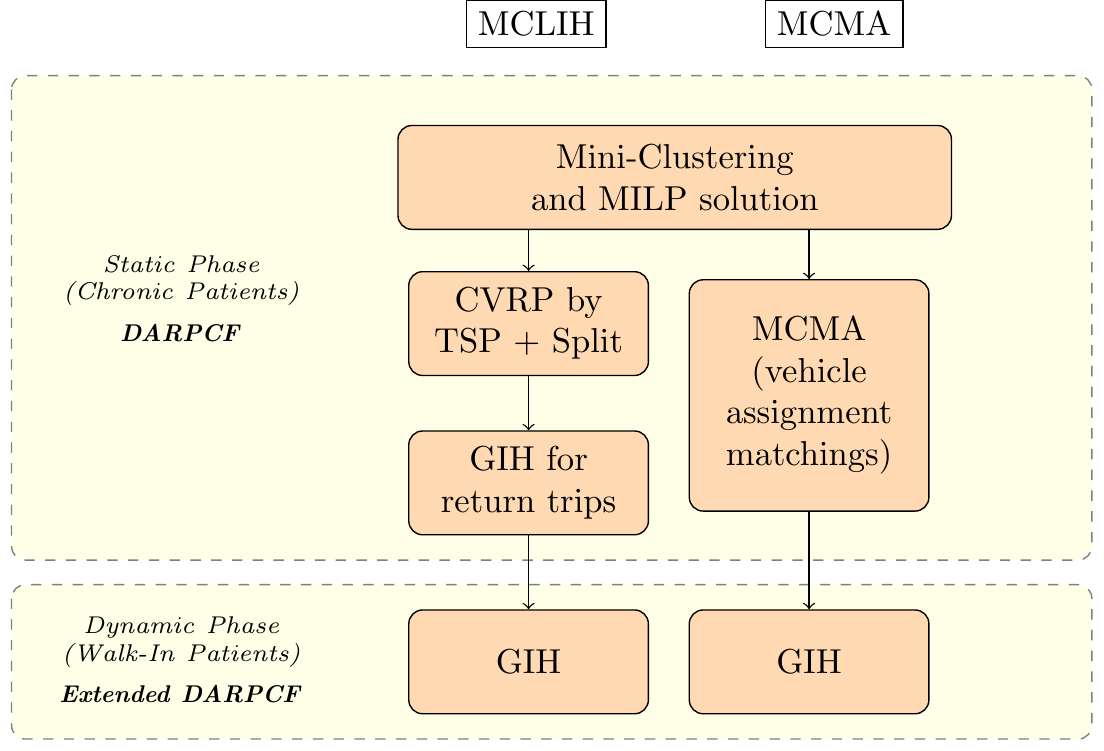}	
	\caption{The solution procedures MCLIH and MCMA.}
	\label{fig:overview}	
\end{figure}

\subsection{Mini-Clustering}
\label{sec:miniclustering}
We start by introducing the notion of a \textit{mini-cluster} which is a set of requests meant to be served by a single vehicle such that the vehicle is empty before and after serving the requests and such that no other request is served by that vehicle in this period.
Therefore, it is a subset of a \textit{cluster}, which generally refers to the set of all requests served by a single vehicle. 
Thus, a DARP vehicle route is composed of routes on mini-clusters and empty rides connecting them.
We denote the cost of the single-vehicle route on a mini-cluster $ \mathcal{M} \in \mathcal{V} $ by $ c(\mathcal{M}) $.

The motivation to use a mini-clustering approach is to exploit geographical structures which may exist in the considered setting.
In the case of primary care in rural areas, we assume that many requests start at the patient's home near a main road and end at a GP that is located in a (sub-)urban hub.
Therefore, many requests may share large parts of their route, which is a key property of efficient mini-cluster routes.

We define a measure of how close two outbound requests $ r_i, r_j \in R_c $ are by considering a complete directed graph $ G'=(V,E) $:
we create one vertex $ v \in V $ for each outbound request $ r \in R_c $ and introduce arc costs $ c'_{ij} \in \mathbb{R}_{\geq 0} $ according to the travel time of the shortest possible way to serve the requests $ r_i $ and $ r_j $ by starting at the pick-up location $ i $ of request $ r_i $.
The three possible paths in $ G $ which serve the requests are $ P^1_{ij} = \{i,n+i,j,n+j\}$, $ P^2_{ij} = \{i,j,n+i,n+j\} $ and $ P^3_{ij} = \{i,j,n+j,n+i\}  $.
For each path, we compute the length $ t(P^k_{ij}) = \sum_{e \in A(P^k_{ij})} t_e $ for $ k=1,2,3 $ and set $ c'_{ij} $ to the minimum, i.e.
$ c'_{ij} := \min\{	t(P^1_{ij}), t(P^2_{ij}), t(P^3_{ij}) \}. $
Thus, $ c'_{ij} $ and $ c'_{ji} $ are usually not equal and therefore the edge costs are asymmetric.

Here, we can profit from the flexible scheduling concept to a great extent.
In the classical DARP setting, two requests that lie close to each other are frequently not compatible due to their time windows, which implies that the preceding arc cost definition $ c' $ would not be useful.
In our setting, we now partition the vertices of the auxiliary graph $ G' $ such that each subset corresponds to a \textit{profitable} mini-cluster.
A mini-cluster in this situation is considered profitable if there are less than $ Q $ nodes in the mini-cluster and if the two requests $ r_i $ and $ r_j $ are sufficiently close to each other.
Precisely, for some proximity parameter $ \rho > 0 $, it must hold that $ c'_{ij} \leq \rho \cdot (t_{i,n+i}+t_{j,n+j}) $.
For example, by choosing $ \rho = 1 $, only arcs would be considered where the shared ride would be shorter than the sum of the respective direct travel times.
By using the vehicle capacity $ Q $ as an upper limit, the capacity constraints for the routes within each mini-cluster become obsolete, which improves the runtime considerably.

We use a simple greedy algorithm similar to Kruskal's algorithm to compute the mini-clusters; see Algorithm \ref{alg:greedy-partition}.
In order to obtain an efficient data structure, we adapt the efficient disjoint-set (also called union-find) data structure introduced by Galler and Fisher \cite{Galler1964} in 1964.
It provides the functions MAKE-SET($ v_i $), FIND-SET($ v_i $) and UNION($ v_i,v_j $) for all nodes $ v_i, v_j \in V $.
MAKE-SET initializes the single-element trees, UNION($ v_i,v_j $) unites the respective trees that $ v_i $ and $ v_j $ belong to and FIND-SET($ v_i $) simply returns the tree which the vertex $ v_i $ belongs to.
\begin{algorithm}[tb]
	\caption{Mini-Clustering in DARP without Time Window Constraints
	}
	\label{alg:greedy-partition}
	\begin{algorithmic}
		\STATE \textbf{Input:} $ G'=(V,E) $ with travel times $ t_{ij} $, vehicle capacity $ Q \in \N $, parameter $  \rho > 0$
		\FOR{$ (v_i,v_j) \in E $}
		\STATE Compute path lengths $ t(P^1_{ij}), t(P^2_{ij}), t(P^3_{ij}) $
		\STATE Compute distance $ c'_{ij}=\min\{t(P^1_{ij}), t(P^2_{ij}), t(P^3_{ij}) \} $
		\ENDFOR		
		\FOR{$ v_i \in V $}
		\STATE MAKE-SET($ v_i $)
		\ENDFOR
		\STATE $ E' := \{ (v_i,v_j) \in E \;\mid\; c'_{ij}\leq \rho \cdot (t_{i,n+i}+t_{j,n+j})  \} $
		\FOR{$ (v_i,v_j) \in E' $ ordered non-decreasingly w.r.t. $ c'_{ij} $}
		\IF{FIND-SET($ v_i $)$ \neq $FIND-SET($ v_j $) 
			\AND $ | $FIND-SET($ v_i $)$| + |$FIND-SET($ v_j $)$| \leq Q $ }
		\STATE UNION($ v_i,v_j $)
		\ENDIF
		\ENDFOR
		\RETURN $ \{ \text{FIND-SET}(v_i) \;|\; v_i \in V \} $
	\end{algorithmic}
\end{algorithm}

The algorithm examines arcs in order of ascending cost $ c' $ and checks for each arc $ (v_i,v_j) $ if, after connecting the two trees containing $ v_i $ and $ v_j $, the tree UNION($ v_i,v_j $) yields a profitable mini-cluster.
In this case, the trees are united and the next arc is considered.
The resulting partition is described by the trees in $ \{ \text{FIND-SET}(v_i) \;|\; v_i \in V \} $.

\subsection{Optimal Routes for Mini-Clusters by Solving MILP}
\label{sec:MILP}
Next, we compute for each mini-cluster an optimal route satisfying all patient requests within the considered cluster.

The advantage of solving only a sub-problem as a mixed integer linear programming problem (MILP) is that some constraints can be omitted.
In comparison to the standard DARP formulation we can omit the capacity and maximum route length constraints since the capacity constraint is checked in the preceding Algorithm~\ref{alg:greedy-partition} and the maximum route length will be checked afterwards when we link the mini-clusters (see Section \ref{sec:TSP+Split}).
Moreover, no time windows need to be respected since the mini-clusters only consist of flexible outbound trips.
The resulting problem is the classical Pickup and Delivery Problem (PDP, also referred to pickup-delivery traveling salesman problem \cite{KALANTARI1985377}) with an additional maximum user ride time constraint.

We formulate this problem as an extension of the open TSP formulation \eqref{openTSP1}~--~\eqref{eq:subtourElim} which is discussed, e.g.~by Parragh et al.~\cite{Parragh2008}.
Let $ \mathcal{M} \in \mathcal{V} $ be a mini-cluster consisting of $ |\mathcal{M}| = n' $ requests and $ G = (N,A) $ the corresponding DARP road graph as defined in Section \ref{sec:DARPform}.
Let $ A(S,\overline{S}) = \{(i,j)\in A : i\in S, j \notin S \} $ be the set of outgoing arcs for any subset $ S \subseteq N $.
For each arc $ (i,j)\in A $, let the binary variable $ x_{ij} $ indicate if the vehicle travels from node $ i $ to node $ j $.
Moreover, the time variable $ B_i$ for $ i \in N $ yields the beginning of the service at node $ i $ and is used for modeling the precedence and user ride time constraints.
Altogether, we obtain the following formulation:
\begin{align}
& \min &\hspace{-5em}  \sum_{(i,j)\in A} c'_{ij} x_{ij}&&\label{openTSP1}\\
& \text{~s.t.~} & \sum_{i: (i,j)\in A}		  x_{ij}& = 1   &\forall& j \in N\setminus \{0\} \label{openTSP2}\\
&				& \sum_{j: (i,j)\in A}		  x_{ij}& = 1   &\forall& i \in N\setminus \{2n'+1\}\label{openTSP3}\\
&				& \sum_{(i,j)\in A(S,\overline{S})}x_{ij}& \geq 1&\forall& S\subseteq N\setminus \{2n'+1\},S\neq \emptyset \label{eq:subtourElim}\\
&				&							x_{i0}& = 0 &\forall& i \in N \label{eq:startDepot}\\
&				&							x_{2n'+1,j}& = 0 &\forall& j \in N  \label{eq:endDepot}\\
& &B_{n'+i} -B_i &\geq 0 										&\forall& i \in P,\label{eq:precedence}\\
& &B_{n'+i} - B_i&\leq L_i & \forall& i \in P.\label{eq:UserRideTime}\\
& &(B_i + t_{ij}) \cdot x_{ij}&\leq B_j 							&\forall& (i,j) \in A \label{eq:timeCons}\\
&				&x_{ij}						  &\in \{0,1\}	&\forall & (i,j)\in A \label{eq:lastDARPmilp}
\end{align}
Here, the open TSP formulation consists of the routing cost \eqref{openTSP1} as the objective function, the in-degree \eqref{openTSP2}, out-degree \eqref{openTSP3} and subtour elimination constraints \eqref{eq:subtourElim}.
Equations \eqref{eq:startDepot} and \eqref{eq:endDepot} define the vehicle depots.
The precedence is guaranteed by constraints \eqref{eq:precedence} which state that the pick-up nodes must be visited before the respective delivery nodes.
Constraints \eqref{eq:UserRideTime} reflect the maximum user ride time.
Note that constraints \eqref{eq:timeCons}, which are necessary for the consistency of the time variables, are non-linear
and equivalent to the condition that if $ x_{ij} = 1 $, then $ B_j \geq B_i + t_{ij}$ for all $(i,j) \in A $.
They can be linearized, e.g.~by introducing constants $ M_{ij} $ and replacing $ (B_i + t_{ij}) \cdot x_{ij}\leq B_j $ by $ B_i + t_{ij} - M_{ij} (1-x_{ij})\leq B_j $, which is similar to the Miller-Tucker-Zemlin subtour elimination constraints \cite{Miller1960}.

Bearing in mind that in our setting the mini-clusters do not contain the vehicle depots, note that the nodes $ 0 $ and $ 2n'+1 $ are dummy nodes.
Accordingly, we can set the start time and the respective travel times to zero, i.e.~$ B_0 = 0, t_{0j}= 0 $ and $ t_{j,2n'+1}= 0 \text{ for all } j\in N $.
Apart from that, all travel times $ t_{ij} $ for $ i,j\in P\cup D$ between the pick-up and delivery nodes are positive,
and therefore constraints \eqref{eq:timeCons} imply the subtour elimination constraints \eqref{eq:subtourElim}, thus making them redundant for the resulting system \eqref{openTSP1}~--~\eqref{eq:lastDARPmilp}.

\subsection{TSP \& Split Approach with GIH for Return Trips}
\label{sec:TSP+Split}
Let us now consider the next step in the MCLIH procedure which generates routes consisting of the mini-clusters computed in the preceding section.

We consider the graph $ H = (\mathcal{V}, \mathcal{A})$ where each mini-cluster $ \mathcal{M} \in \mathcal{V} $ corresponds to a vertex and where the arc costs $ c_{ij} \in \mathbb{R}_{\geq 0} $ for $ \mathcal{M}_i,\mathcal{M}_j \in \mathcal{V} $ are set to the travel time from the last drop-off location in the optimal route $ \mathcal{S}_i $ on $ \mathcal{M}_i $ to the first pick-up location in $ \mathcal{S}_j $.
We interpret the problem as a capacitated vehicle routing problem (CVRP) with slightly adjusted constraints compared to the standard formulation \cite{Labadie2016}.

The basic idea of finding vehicle routes $ \mathcal{R}_1, \mathcal{R}_2, \dots, \mathcal{R}_m $ starting from a depot is the same as in the CVRP, however, we observe that our problem is set on a directed graph with asymmetric arc costs.
Moreover, there is no actual capacity constraint to be considered because the vehicles are empty between mini-clusters.
The route duration $ d_{\mathcal{M}} $ within each mini-cluster $ \mathcal{M} $ can be interpreted as service duration at that node.
Thus, the important restriction for each vehicle, the maximum route duration, is of the form
\begin{align}\label{eq:routeLength}
\sum_{e \in \mathcal{A}(\mathcal{R}_k)} c_e +\sum_{i\in \mathcal{V}(\mathcal{R}_k)} d_i \leq L
\end{align}
for each route $ \mathcal{R}_k $ for $ k = 1,\dots,m $.
We can solve the problem in question by adapting CVRP algorithms to consider directed graphs and to respect equation \eqref{eq:routeLength} instead of the usual route length and capacity constraints.

Among the main heuristic concepts for solving CVRPs, there are two somehow contrary approaches: cluster-first route-second methods and route-first cluster-second methods \cite{Labadie2016}.
The latter were extensively studied and reviewed by Prins \cite{PRINS2014179} and have proven to produce competitive solutions, in particular for large instances.
Therefore, we decide to use the classical split method by Beasley \cite{BEASLEY1983403} in a more compact version proposed by Prins \cite{PRINS20041985,Labadie2016}.

Initially, a traveling salesman problem (TSP) on the vertices of $ H $ has to be solved.
To that end, we use an ant swarm TSP heuristic following the approach proposed by Dorigo et al.~\cite{Dorigo2002}.
The next step, splitting the TSP tour into sub-tours can be solved optimally in polynomial time by using a shortest-path algorithm on an auxiliary graph \cite{BEASLEY1983403}.
The only necessary adaptations of the two algorithms to our setting are that the TSP needs to be solved with asymmetric edge costs and that the splitting procedure must respect equation \eqref{eq:routeLength}.
This can be done by modifying the CVRP capacity constraint by interpreting the service duration of a mini-cluster $ \mathcal{M} $ as the load and by including the arc costs in the load calculation.

From the obtained vehicle routes we can deduce the times when the appointments of the chronic patients can start and thus determine the time windows for the inbound trips.
By using the greedy insertion heuristic GIH based on \cite{Jaw1986,Madsen1995} for the inbound trips, we finalize the schedules of the chronic patients.
We assume that the number of vehicles is always large enough, such that this insertion is possible.
However, it is also an option to remove an outbound request from the schedules in case of an unsuccessful insertion of the corresponding inbound trip and try reinsertion for different time windows.

The procedures described in Sections \ref{sec:miniclustering}--\ref{sec:TSP+Split} determine the MCLIH heuristic for chronic patients.
In the next section, we present an alternative procedure to schedule the computed mini-clusters.

\subsection{Matching Approach}
In this section, we propose the Mini-Cluster Matching Algorithm (MCMA) which does not schedule mini-clusters of outbound trips and inbound trips separately as in MCLIH, but rather handles both types of rides simultaneously.
The motivation behind this approach is to properly make use of the information that each scheduled outbound trip leads to an inbound trip.

In MCMA, we successively build vehicle schedules from the beginning of the service period until the end in a rolling horizon approach.
We initialize each vehicle schedule with a mini-cluster of outbound trips computed in Section \ref{sec:miniclustering}, and consider a new planning horizon every time a vehicle drops off a passenger.
For each planning horizon corresponding to the drop-off of an outbound trip,
we must guarantee that after the passenger's stay at the GP of length $ d_{\text{GP}} $, a vehicle picks up the passenger within a certain time window for the inbound trip.
These decisions can be modeled as a bipartite matching problem between vehicles and jobs where edges exist when time windows are respected.
The edge weights reflect travel times between vehicle locations and job locations, and if the vehicle has to wait at the job location due to time windows.
The resulting matching is used to create a so-called \textit{interim schedule} for each vehicle, i.e.~a continuation of the previous (fixed) schedule which realizes the matching, but which is possibly still subject to change in later time horizons.
If a vehicle becomes empty and reaches the end of its fixed schedule, the current interim schedule is used for updating the fixed schedule.
Moreover, when a new inbound request is added after an outbound passenger delivery, we check if the inbound can be inserted into an interim schedule instead of computing a new matching.

\subsubsection{Vehicle Rescheduling by Bipartite Matching Problems}

We now consider special bipartite matching problems which are suited to this application by taking vehicle positions into consideration and differentiating between the two request types.
The matching problems are characterized by a vertex set $ V=\{v_1,\dots,v_m\} $ of vehicles and a vertex set $ W = W_1 \dot{\cup} W_2 $ of transportation requests which is composed of a set $ W_1 $ of inbound trips that \emph{must} be matched within a certain time limit, and a set $ W_2 $ of outbound trips that \emph{can} be matched.

More precisely, for a time horizon corresponding to time $ t $, we consider a complete bipartite graph $ G = (V\dot{\cup} W, E) $ where
the transportation requests $ W=W_1 \dot{\cup} W_2 $ are described by
the open inbound requests $ W_1 \subseteq \overline{R_c}=\{\overline{r_1},\dots,\overline{r_n} \} $ (i.e.~patients which need to be picked up at the GP within the associated time window $ [e_{\overline{i}},l_{\overline{i}}] $ for $ \overline{r_i} \in W_1 $ such that $ t \leq l_{\overline{i}} $),
and by the open outbound requests $ W_2 \subseteq \mathcal{V} = \{ \mathcal{M}_1,\dots,\mathcal{M}_{|\mathcal{V}|} \} $ (partitioned into the mini-clusters introduced in Section \ref{sec:miniclustering}).
As before, the outbound requests do not have a pre-specified time window.
When assigning a mini-cluster $ \mathcal{M}_j \in W_2 $ to a vehicle, we use the intra-cluster schedule $ \mathcal{S}_j $ computed in Section \ref{sec:MILP}.

The weight of an edge between a vehicle and a transportation request in $ G $ reflects the transition time which the vehicle needs between its last preceding job and the start of the considered request.
We denote by $ d_i $ the remaining time on the ride of vehicle $ v_i \in V $ at time $ t $.
Moreover, $ t_{ij} $ denotes the travel time from the vehicle's previous delivery to the pick-up node of inbound request $ \overline{r_j} $ or the starting node of mini-cluster $ \mathcal{M}_j $.
In case of an inbound request 
$ \overline{r_j} $, we have to set the cost to $ +\infty $ if the vehicle would arrive after the end of the feasible time window,  i.e.\  $ t + d_i+ t_{ij} > l_{\overline{j}} $.
Otherwise, we set the cost to the maximum of $ t_{ij} $ and the amount of time which the vehicle has to wait at the node until the start of the time window.
For a mini-cluster of outbound requests, it is simply set to $ t_{ij} $:

\begin{align}\label{eq:matchingEdges}
c_{ij} = \begin{cases}
\max\{t_{ij},e_{\overline{j}}-t-d_i\},		& \mbox{ if } v_i \in V, \overline{r_j} \in W_1 \mbox{ and } t + d_i+ t_{ij} \leq l_{\overline{j}},\\
+\infty, 						& \mbox{ if } v_i \in V, \overline{r_j} \in W_1 \mbox{ and } t + d_i+ t_{ij} > l_{\overline{j}},\\
t_{ij},							& \mbox{ if } v_i \in V, \mathcal{M}_j \in W_2,
\end{cases}
\end{align}

We call the problem of finding a matching in the described graph such that every vehicle is assigned to a ride in a fashion that every inbound trip is matched the \textit{vehicle rescheduling problem}.
Moreover, we refer to the computed matching as a \textit{vehicle assignment matching}.
A vehicle assignment matching of minimum cost minimizes the total time that vehicles have to drive or wait until the next ride.

In the following we show that if the inbound requests in $W_1$ can be matched to the vehicles with finite weight, the computation of an optimal solution to the vehicle rescheduling problem can be computed in polynomial time.
We thereby exploit that the (classical) minimum-weight maximum matching problem in bipartite graphs is polynomial-time solvable, e.g.\ by the Hungarian method (also Kuhn-Munkres algorithm) \cite{Edmonds1972}.

\begin{theorem}
	Let $ G $ denote the bipartite vehicle rescheduling graph defined above and let $ G[V\cup W_1] $ denote its restriction to the vehicles $ V $ and the inbound requests $ W_1 $.
	If a matching $ M^{W_1} $ on $ G[V\cup W_1] $ with size $ |W_1| $ and finite weight 
	exists, then the vehicle rescheduling problem can be solved optimally in polynomial time, i.e.~a maximum matching on $ G $ can be found which has minimum weight among all maximum matchings on $ G $ that match every vertex in $ W_1 $.
\end{theorem}
\begin{proof}
	We prove the theorem by showing that we can transform the vehicle rescheduling problem to a classical minimum-weight maximum matching problem in a bipartite graph.
	
	Without loss of generation, we assume that $ |W_1| \leq |V| \leq |W| = |W_1| + |W_2| $ holds, i.e. there are more vehicles than inbound requests (since matching $ M^{W_1} $ exists) but less vehicles than the sum of both types of requests (since otherwise the problem simplifies to a standard min-weight maximum matching).
	
	Let $ \overline{C} > 0 $ be a constant larger than all finite weights of edges in $ G $, i.e.~it holds
	$ \max \{c_e \mid e \in E, c_e < \infty \} < \overline{C}$. 	
	We define alternative edge weights $ c'_{ij} $ for a vehicle $ v_i \in V $ and a transportation request $ w_j \in W $ via
	\[ c'_{ij} = \begin{cases}
	c_{ij}, 			& \mbox{ if } i \in V, j \in W_1,\\
	c_{ij} + |W| \cdot \overline{C},& \mbox{ if } i \in V, j \in W_2.
	\end{cases} \]
	Let $ M^* $ denote an optimal solution to the vehicle rescheduling problem, i.e.~a minimum weight matching in $G$ with respect to $ c $ that matches all vertices in $ V $ such that all vertices in $ W_1 $ are matched.
	Moreover, let $ M' $ be an optimal classical minimum-weight maximum matching in $G$ with respect to $ c' $.
	Note that $ |M^*| = |M'| = |V| $ holds since $ |M^*| $ matches all vehicles and $ |M'| $ must be of at least the same size because it is a maximum matching.
	
	We show by contradiction that $ M' $ matches all vertices in $ W_1 $ and has the same weight as $ M^* $ with respect to $ c $, i.e.~	$ c(M^*) = c(M'). $
	To that end, assume that $ M' $ does not match all vertices in $ W_1 $.
	Then $M'$ must match more vertices in $ W_2 $ than $ M^* $ as $|M^*| = |M'|$.
	Let $ M_{|W_i} $ for $ i = 1,2 $ denote the restriction of a matching to a vertex subset, i.e.
	\[ M_{|W_i} := \{e\in M \mid \exists w \in W_i \text{ s.t. } w\in e\}.  \]
	Thus by setting $ k^* := |M^*_{|W_2}| $ and $ k' := |M'_{|W_2}| $, we get that $ k^* < k' $.
	From $ c(M^{W_1}) < \infty $ and the finite weight of edges between $ V $ and $ W_2 $, we can conclude that both matchings have finite weight, in particular $ c(M^*_{|W_1}) < |W_1| \cdot \overline{C} $ and $ c(M^*_{|W_2}) < k^* \cdot \overline{C} $.	
		It follows that	
		\begin{align*}
		c'(M^*) &= c'(M^*_{|W_1}) + c'(M^*_{|W_2}) \\
		&= c(M^*_{|W_1}) + c(M^*_{|W_2}) + k^* \cdot |W| \cdot \overline{C}\\
		&< |W_1| \cdot \overline{C} + k^* \cdot \overline{C} + k^* \cdot |W| \cdot \overline{C}\\
		&< |W_1| \cdot \overline{C} + |W_2| \cdot \overline{C} + k^* \cdot |W| \cdot \overline{C}\\				
		&= (1 + k^*) \cdot|W|  \cdot \overline{C}\\
		&\leq k' \cdot |W|  \cdot \overline{C}\\
		&\leq c(M'_{|W_1}) + c(M'_{|W_2}) + k' \cdot |W| \cdot \overline{C}\\
		&= c'(M')				
		\end{align*}
		due to the definition of $ c' $, which is a contradiction to the assumption that $ M' $ is a minimum weight matching in $G$ with respect to $c'$.
		Hence, by contradiction, $ M' $ matches all vertices in $ W_1 $ and is therefore also an optimal solution to the vehicle rescheduling problem with respect to $ c' $.
		Since $ c $ only differs from $ c' $ by the constant $ |W| \cdot \overline{C} $ added to the cost of each edge between $ V $ and $ W_2 $, $ M' $ is not only an optimal solution with respect to $ c' $, but also with respect to $ c $.
		This shows that the vehicle rescheduling problem can be solved by computing a standard minimum-weight maximum matching problem with respect to the alternative edge costs $ c' $ and the statement follows since the Hungarian method runs in $ \mathcal{O}((|V|+|W|)^3) $ time \cite{Edmonds1972}.
\end{proof}

With the tools at hand to solve the vehicle rescheduling problem, we can now introduce the MCMA procedure.

\subsubsection{The MCMA Procedure}
In this section, we use the notation introduced in Section \ref{sec:DARPform}, i.e.~let $ G $ denote the DARP road graph.

According to our initial motivation, we solve the vehicle rescheduling problem every time a vehicle becomes empty.
Vehicles are then matched either to a mini-cluster of outbound trips or to a single inbound request.
Note that mini-clusters of outbound trips utilize the vehicle very efficiently, but single inbound requests leave vehicle capacity unused.
In order to tap this potential, the MCMA framework provides a method for increasing vehicle utilization by additional means.
Precisely, we do not include an upcoming inbound request directly in a new matching, but first check whether it can be inserted efficiently into an already existing interim schedule.

Before describing the procedure in detail, let us introduce some notation which is mainly used in the flowchart depicted in Figure \ref{fig:MCMAprocedure}.


The key structure which is used for realizing the rolling horizon approach is a set $ \mathcal{Q} $ which contains the delivery times of all patients who are currently on a vehicle.
Precisely, it is a priority queue from which the reference times for each time horizon are extracted.
In iteration $ i $, the next delivery is considered, i.e.~we select $ t_i $ from $ \mathcal{Q} $ such that $ t_i = \min_{t \in \mathcal{Q}} t $.
\par
\begin{wrapfigure}{R}{0.40\textwidth}
	\vspace*{-0.31cm}
	\resizebox{!}{0.7\textheight}{%
		\includegraphics{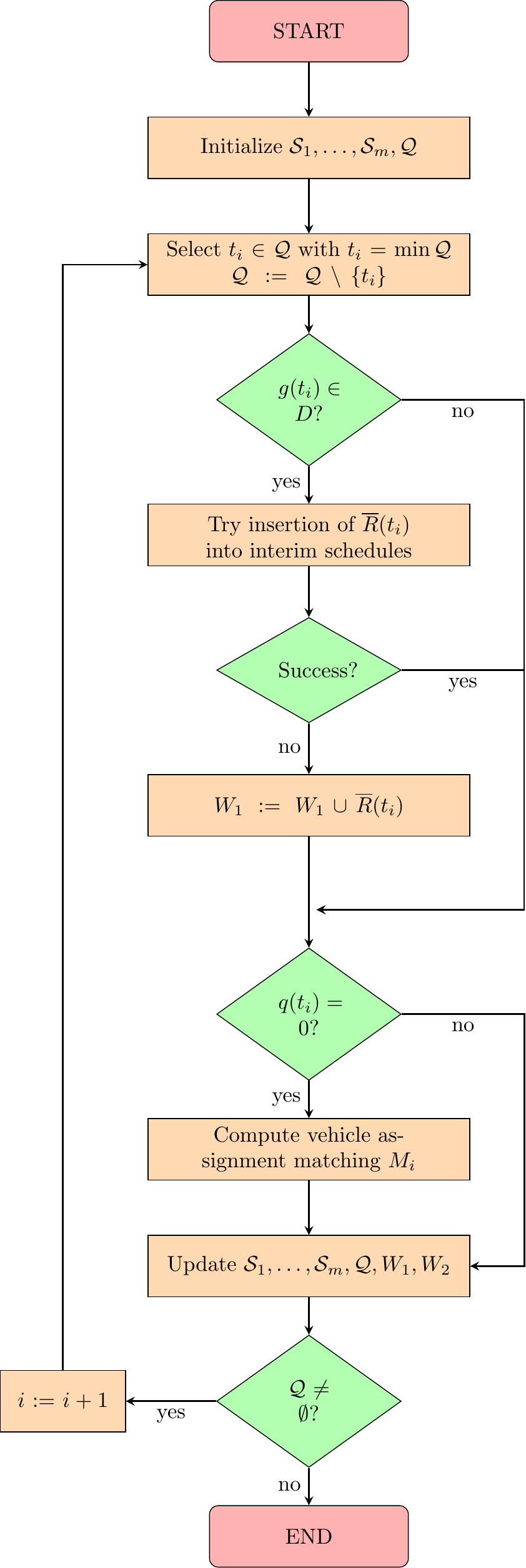}
	}
	\setlength{\belowcaptionskip}{-1.7cm}
	\caption{Flowchart MCMA for DARPCF}
	\label{fig:MCMAprocedure}
\end{wrapfigure}
For each time $ t_i $, we assign the following values:
\begin{itemize}
	\item $ R(t_i) \in R_c\cup \overline{R_c} $ denotes the request which corresponds to the delivery of the patient at time $ t_i $.
	If it is an outbound request, i.e.~$ R(t_i) \in R_c $, then the corresponding inbound request is denoted by $ \overline{R}(t_i)\in \overline{R_c} $
	\item $ q(t_i) $ is the remaining load on the vehicle after dropping off the patient belonging to request $ R(t_i) $.
	\item $ W_1 $ (or $ W_1(t_i) $) is the set of inbound requests which need to be matched because the patient was dropped off before time $ t_i $ and is not yet assigned definitely to a return ride (so-called \textit{waiting list}).
	\item $ W_2 $ (or $ W_2(t_i) $) contains the mini-clusters of outbound requests that have not yet been matched definitely to a vehicle.
\end{itemize}

The procedure starts by choosing a mini-cluster for each vehicle, which is used to initialize the schedules $ \mathcal{S}_1, \dots, \mathcal{S}_m $ and the set of drop-off times $ \mathcal{Q} $.
This choice can be made randomly or according to some other heuristic, e.g.~mini-clusters with longest total duration first.
Then, we enter the loop where, in each iteration, we extract the next drop-off time $ t_i $ from $ \mathcal{Q} $ and consider the corresponding time horizon.

First, we check if the delivered patient had an inbound ride, i.e.~$ R(t_i) \in R_c $ or an outbound ride, i.e.~$ R(t_i) \in \overline{R_c} $.
In case of an outbound ride, we try to insert the corresponding inbound ride $ \overline{r_i} := \overline{R}(t_i) $ into the interim schedules.
An insertion is considered successful if the time windows are respected and if at least one other request in the schedule is ``close'' to $ \overline{r_i} $ with respect to the same criteria that we used for the mini-clusters in Section \ref{sec:miniclustering}.
Otherwise, $ \overline{r_i} $ is added to the waiting list $ W_1 $.

Then, we check the vehicle load $ q(t_i) $.
If the last passenger is dropped off (i.e.~$ q(t_i) = 0 $), we compute a matching between the vehicles and the requests in $ W_1 $ and $ W_2 $.
Otherwise, if $ q(t_i) \neq 0 $, we wait until the next drop-off before assigning a new ride.
Therefore, there is no need for a matching at that point.
Subsequently, the rides in interim schedules of empty vehicles or vehicles where a ride has been inserted are shifted to the corresponding definitive (fixed) schedules $ \mathcal{S}_1, \dots, \mathcal{S}_m $ and the drop-off times of the assigned requests are inserted into $ \mathcal{Q} $.
This is the end of a planning horizon and the procedure continues with the next drop-off in iteration $ i+1 $.

Let us close this section with some further considerations to MCMA.

\paragraph{Avoiding Failures}
There is one major failure, which is that a scheduled patient cannot be served for the return trip within its time window.

This problem of unassigned inbound rides is most likely to occur when the vehicle capacity and utilization is high and the fleet size is small.
In these cases, it might happen that there are more people waiting to be picked up than vehicles available to serve them, which means that the outbound ride should not have been scheduled in the first place.
For example, this may happen if unfavorable initial rides with similar drop-off times are chosen.
In MCMA, we remove outbound requests from the schedule if we failed to serve the respective inbound trip.
Then, we try to reinsert the outbound trip into the current interim schedules by using GIH. 
If the insertion was successful, the respective interim schedule is turned into a definitive schedule and the new drop-off times are added to $ \mathcal{Q} $.
Otherwise, the ride must be added as a single-ride mini-cluster to the remaining set of mini-clusters $ W_2 $.

At the end of the service, i.e.~when the vehicles exceed their maximum route length, open mini-clusters in $ W_2 $ must be denied by the operator.
This should be avoided in practice by limiting the number of accepted requests to an estimate of how many requests can be served.

\paragraph{Runtime}
The number of iterations is bounded by the number of delivery times which are added to the set $ \mathcal{Q} $.
This happens to each of the $ n $ requests at most twice; once for the outbound and once for the inbound trip.
The loop itself is dominated by the matching which is in $ \mathcal{O}((|V|+|W|)^3) $.
In total, we obtain a number of iterations in $ \mathcal{O}(n^4) $ since we can assume that $ |V| \leq n $ and $ |W| \leq n $ holds.

\subsection{Dynamic Insertion of Online Requests (Extended DARPCF) and Further Adaptations}
After having scheduled the chronic patients, we serve as many of the walk-in requests $ R_w \cup \overline{R_w} $ as possible.
If a request arises after the beginning of the service of the respective day, GIH must insert the request dynamically (i.e.~online).
However, by interpreting the locations of the vehicles at the time of the request as vehicle depots and adjusting the vehicle loads and ride time constraints, the offline version of the insertion heuristic can be used without extensive modifications.
The insertions are always tried pairwise, i.e.~either both outbound and inbound trip can be inserted or the request is denied.

\begin{table}
	\begin{tabulary}{\textwidth}{LLL}
		\toprule
		\textbf{Challenge} & \textbf{MCLIH} & \textbf{MCMA} \\ \midrule
		\textbf{Taking non-continuous opening hours into account (i.e.~include lunch breaks)}&
		Modify splitting procedure so that it introduces breaks&
		When rolling horizon reaches begin of lunch break, match open inbound requests, insert break and match new outbound trips after the break\\
		\hline
		\textbf{Limiting the number of patients who arrive at a GP at the same time (avoid congestions at GP)}&
		Before fixing an interim schedule, check if a congestion would arise, otherwise delay the trip (insert gap)&
		Before starting the insertion of inbound trips, check for congestions and reinsert requests causing congestions by using GIH\\
		\hline
		\textbf{Uniform use of whole service period/leaving space for insertion of walk-ins and inbound trips}&
		Split procedure results in evenly distributed requests over service period; reducing the number of necessary vehicles possible by increasing cost of depot rides&
		Insert gaps when percentage of served requests exceeds the percentage of time which has passed in rolling time horizon\\
		\bottomrule
	\end{tabulary}
	\caption{Handling of further challenges arising in realistic setting}
	\label{tab:challenges}
\end{table}

Before the presented algorithms can be applied to the real context, a few additional considerations should be made.
We shortly describe the arising challenges and the way we handle them in case of both algorithms in Table \ref{tab:challenges}.

\section{Computational Results}
\label{sec:CompStudy}
In this section, we describe the results of the study which we performed with our algorithms MCMA and MCLIH.
The aim is to verify whether or not the idea of flexible scheduling can improve a dial-a-ride system in primary care.
Therefore, we focus on comparing the results in a flexible scheduling setting with the standard setting rather than applying our algorithms to other DARP instances from the literature.
In particular, we describe in Section \ref{sec:dataset} how we generate instances which are closely related to primary care in rural areas.
As a reference for the standard setting we use GIH, the greedy insertion heuristic inspired by Madsen's REBUS algorithm \cite{Jaw1986,Madsen1995} which is also part of MCLIH and MCMA for the insertion of walk-in patients.

\subsection{Data Generation and Study Design}
\label{sec:dataset}
We evaluate the presented flexible scheduling approach based on a real-world primary care system in Germany.
The system comprises three predominantly rural municipalities and features $20$ GPs with health insurance accreditation that care for $35542$ inhabitants.
As empirical transportation requests are unavailable, we resort to simulation to generate artificial requests that reflect the system's structural characteristics.
In particular, we make use of the existing model of the considered primary care system in the hybrid agent-based simulation tool SiM-Care~\cite{Comis19}.
SiM-Care models both patients and GPs as individual agents and tracks their micro-interactions.
As a result, we obtain the scheduled appointments of chronic and acute patients as well as the visits of walk-in patients for a one-year horizon.
Combining these with the locations of patients and GPs, we can generate the required outbound and inbound transportation requests that serve as the input to our models.
We group these requests into instances that correspond to one day of simulated GP services.
The number of pairwise requests per instance varies between 1100 and 1350 on days without afternoon session and between 1600 and 2000 on days with afternoon session.
The percentage of chronic requests lies between $ 10\% $ and $21\%$; see Table \ref{tab:Instances} for the sample of 20 instances corresponding to an arbitrarily chosen period of four weeks used in the study.
The distances between the locations are based on road network data and computed by using the Open Source Routing Machine (OSRM) \cite{luxen-vetter-2011}.
\begin{table}
	\centering
	\begin{tabular}{rrrrr}
\toprule
 Data Set &  Total &  Chronic &  		 Walk-In &  Chronic/Total \\
\midrule
    21903 &   1996 &      362 &                     1634 &       0.181363 \\
    21904 &   1640 &      190 &                     1450 &       0.115854 \\
    21905 &   1181 &      120 &                     1061 &       0.101609 \\
    21906 &   1752 &      226 &                     1526 &       0.128995 \\
    21907 &   1318 &      138 &                     1180 &       0.104704 \\
    21910 &   1835 &      340 &                     1495 &       0.185286 \\
    21911 &   1633 &      196 &                     1437 &       0.120024 \\
    21912 &   1102 &      128 &                      974 &       0.116152 \\
    21913 &   1718 &      184 &                     1534 &       0.107101 \\
    21914 &   1309 &      140 &                     1169 &       0.106952 \\
    21917 &   1943 &      362 &                     1581 &       0.186310 \\
    21918 &   1679 &      212 &                     1467 &       0.126266 \\
    21919 &   1163 &      170 &                      993 &       0.146174 \\
    21920 &   1594 &      190 &                     1404 &       0.119197 \\
    21921 &   1278 &      148 &                     1130 &       0.115806 \\
    21924 &   1893 &      398 &                     1495 &       0.210248 \\
    21925 &   1662 &      204 &                     1458 &       0.122744 \\
    21926 &   1203 &      136 &                     1067 &       0.113051 \\
    21927 &   1639 &      184 &                     1455 &       0.112264 \\
    21928 &   1246 &      138 &                     1108 &       0.110754 \\
\bottomrule
\end{tabular}

	\caption{Patient structure in the considered instances}
	\label{tab:Instances}
\end{table}

\paragraph{Study Design}
As described in the previous paragraph, we consider $20$ instances, each comprising more than $1100$ patient requests.
Since each patient request is due to a required consultation with a GP, it eventually results in one outbound trip and one inbound trip which have to be scheduled.
In our study, we compare the DARPCF algorithms in the flexible setting with the standard setting.
For MCMA and MCLIH we discard the appointment times of chronic patients and re-schedule them throughout the whole service period.
For GIH in the standard setting, we keep the original appointment times and create time windows before and after the appointment for the outbound trip and inbound trip, respectively.
Recall that GIH is also part of MCMA and MCLIH for the insertion of walk-in patients, therefore the main difference between the two settings is how many of the walk-in patients can still be inserted into the schedules after the static phase.

In our study, we use the default values for the parameters depicted in Table \ref{tab:DefaultValues}.
Moreover, we perform sensitivity analyses where single parameters like the maximum user ride time and vehicle capacity are varied to better understand the behavior of the algorithms.

\begin{table}
	\centering
	\begin{tabulary}{\linewidth}{ccL}
		\toprule
		Parameter &  Value &  Information\\
		\midrule
		$ m $			& 10		& Fleet Size\\
		$ Q $			& 4			& Vehicle Capacity\\
		$ W $			& 20 min.	& Maximum time window length\\
		$ d_{\text{GP}} $	& 30 min.	& Expected duration of each stay at the GP\\
		$ L_i $		& $ 1.5 \cdot t_{i,n+i} $ & The maximum user ride time (proportional to direct travel time)\\
		$ \rho $		& 1.5		& Measure of proximity of requests\\
		$ (C,W)_{GP} $		&$  (6,30\text{ min.}) $&GP congestion parameters, e.g.~each 30 min. no more than 6 patients can be received\\
		\bottomrule	
	\end{tabulary}
	\caption{Default parameter choices}
	\label{tab:DefaultValues}
\end{table}

\paragraph{Implementation}
All algorithms presented in this work were implemented in Java (using OpenJDK 10) and the experiments were performed on an Ubuntu 18.04.1 LTS system (64 Bit) with an Intel Xeon W3540 CPU clocked at 2.93GHz.

We use the following external libraries or source codes:
The routes between the different locations are calculated by using OSRM \cite{luxen-vetter-2011}.
Moreover, for deducing vehicle locations for the insertion of the online requests of walk-in patients, we use the waypoints and the segment durations in the routes provided by OSRM.
For the mini-clustering Algorithm \ref{alg:greedy-partition}, we modified the disjoint-set data structure and the Kruskal algorithm implementation by Esmer \cite{Esmer2017}.
The MILP solution for the intra-cluster route is calculated up to optimality by using the IBM ILOG CPLEX Optimization Studio v12.8 \cite{CPLEX}.
The TSP tour linking the mini-clusters is computed by an ant colony optimization algorithm implemented by Dodd \cite{Dodd2013,Dorigo2002}.
For solving the min-weight maximum matchings we use the Hungarian method implemented in the JGraphT library \cite{jgrapht}.

\subsection{Results}
\label{sec:results}
In this section, we present the results of our computational study.
First, we evaluate which algorithm produces the best solutions in the default setting.
Then, we analyse the algorithms' sensitivity to changes in the default parameter choices.

\paragraph{Main Result}

\begin{figure}
	\begin{subfigure}[b]{\linewidth}
		\centering
		\includegraphics[width=0.7\linewidth]{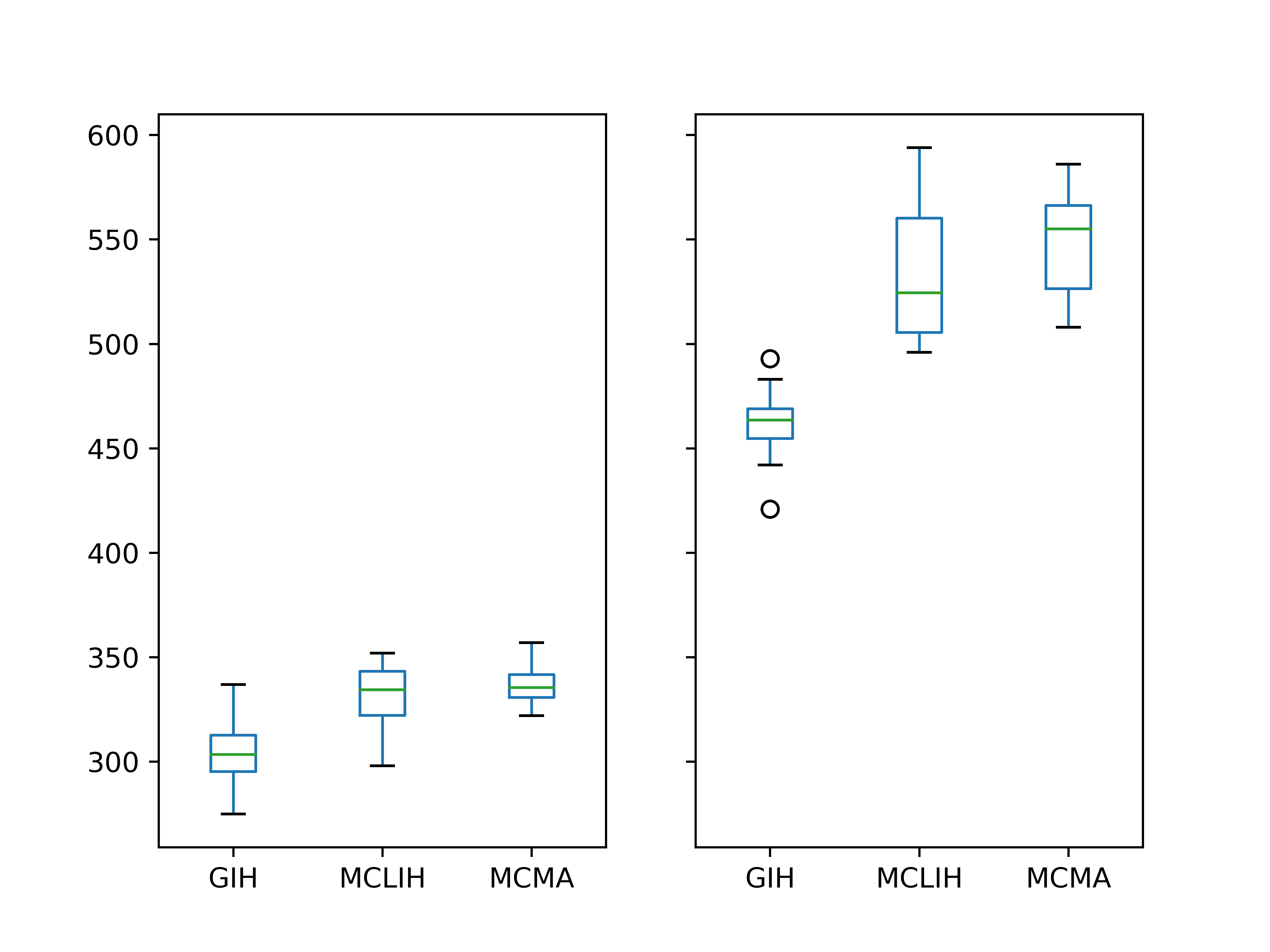}
		\caption{Box plot for the number of served requests for days without (left) and with afternoon session (right)}
		\label{fig:overviewplot1}
	\end{subfigure}
	\begin{subfigure}[b]{\linewidth}
		\centering
		\includegraphics[width=0.7\linewidth]{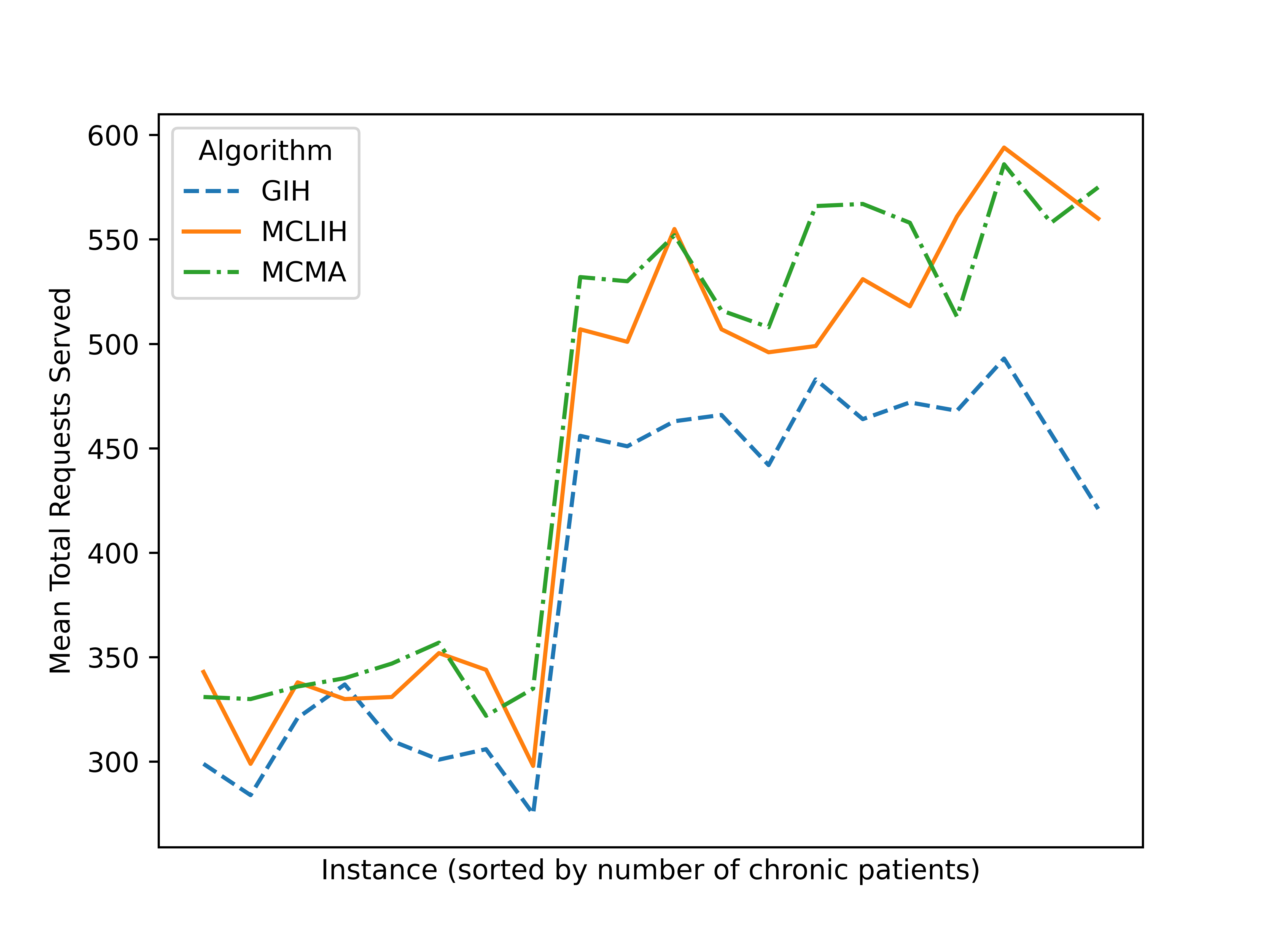}
		\caption{Mean total number of served requests per instance (sorted by number of chronic patients)}
		\label{fig:overviewplot3}
	\end{subfigure}
	\caption{Number of served requests by algorithms MCLIH, MCMA and GIH with default parameters}
	\label{fig:overviewplots}
\end{figure}
Based on the results depicted in Figure \ref{fig:overviewplots}, we can state that the introduced algorithms MCMA and MCLIH that employ the flexible scheduling of chronic patients outperform GIH.
Figure \ref{fig:overviewplot3} shows that with the exception of one instance, both MCLIH and MCMA manage to accommodate substantially more rides than GIH.
In the overall average of the $20$ considered instances, MCMA serves about $460$ requests (mean: $ 462.95 $, median: $ 514.5 $) and MCLIH serves about $450$ requests (mean: $ 452.05 $, median: $ 500 $), whereas GIH serves more than 50 requests less (mean: $ 398.45 $, median: $ 446.5 $).
In fact, by using MCMA and the flexible scheduling system, about $ 16\% $ more requests can be served compared to the default system using GIH.
Note also that this difference is more pronounced on days with an afternoon session, i.e.~MCMA and MCLIH seem to profit from longer service periods and more requests (increase of $ 18.5\% $).
In Figure \ref{fig:overviewplot3}, we can furthermore notice that instances with a higher number of chronic patients can benefit more from the flexible scheduling of chronic requests which seems reasonable.

Let us now perform some sensitivity analyses.

\paragraph{Evaluation of Different Vehicle Capacities}
\label{sec:studyVehCapacities}
We compare the performance of our algorithms for different vehicle capacities with values $ Q \in \{3,4,5,6\} $.
Note that we limit the vehicle size to $ 6 $ as this covers the most common vehicle sizes in patient transportation and keeps the run-time for the
exact MILP solution within the mini-clusters described in Section \ref{sec:MILP} in the range of seconds.
The results are shown in Figure \ref{fig:increasingQ} and we can conclude that an increased vehicle capacity does not generally imply that more requests are served.
\begin{figure}
	\centering
	\includegraphics[width=0.7\linewidth]{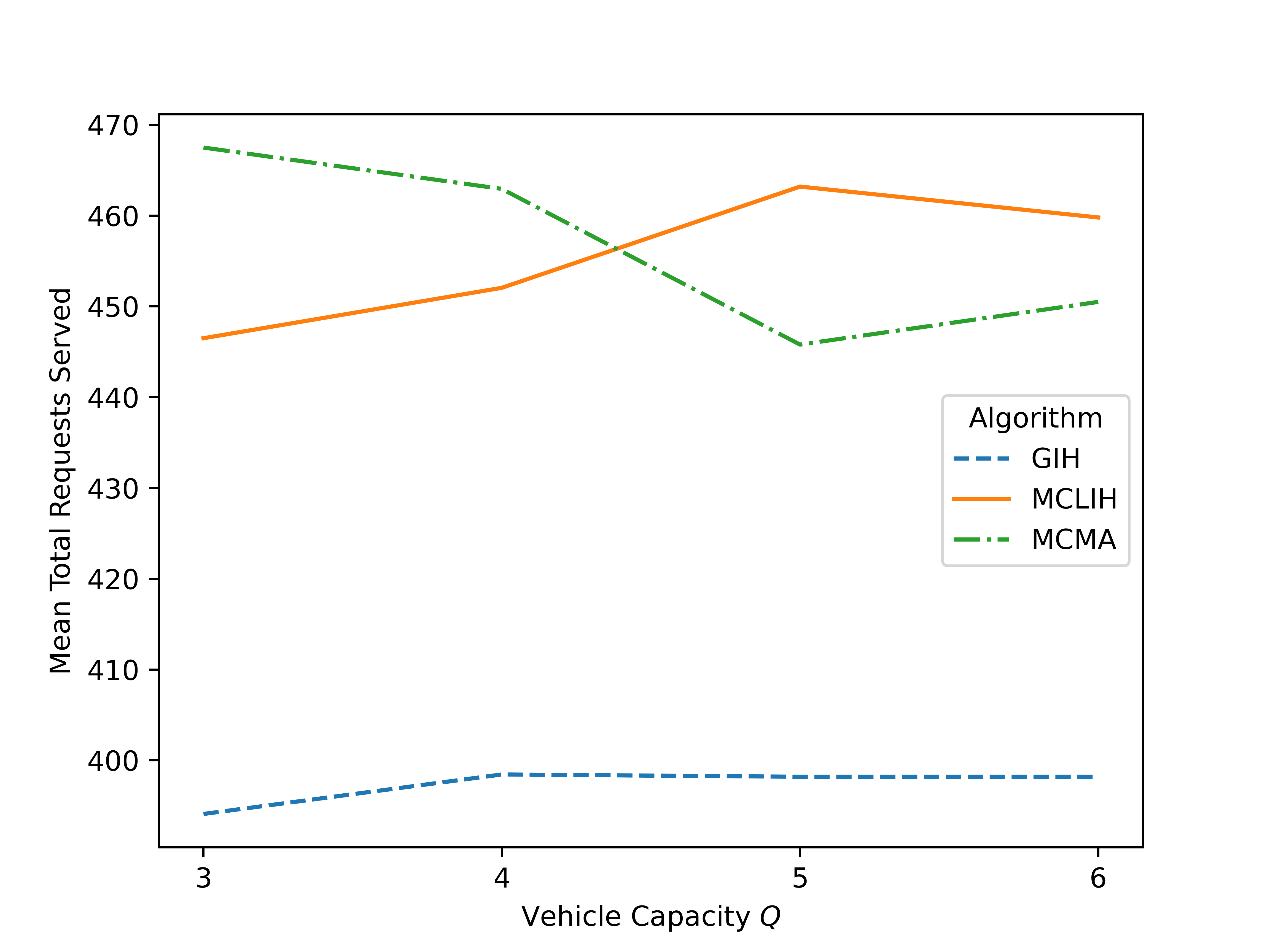}
	\caption{Mean number of total requests served for different vehicle capacities}
	\label{fig:increasingQ}
\end{figure}
We actually note that most requests are served by MCMA for $ Q = 3 $, where the mean over all instances is $ 467.5 $.
This may reflect that for large mini-clusters serving the corresponding inbound trips efficiently is more difficult than for smaller ones.

On the other hand, thorough analysis of the resulting schedules shows that even though we increase the vehicle capacity, it hardly happens that $ 5 $ or more patients are on board of a vehicle at the same time.
This suggests that the imposed capacity limit of $ Q = 6 $ does hopefully not pose a restriction in reality.
We suspect that this is due to the maximum user ride time $ L_i $ being a far more limiting constraint during the schedule creation than the vehicle capacity, which we want to study in the following.

\paragraph{Evaluation of Increasing Maximum User Ride Time}
We investigate how changes in the maximum user ride time $ L_i $ influence the resulting vehicle schedules.
Reasonable choices for $ L_i $ include all values between 1 (i.e.~only direct rides) and 2 (rides can take twice as long as the direct ride).
Figure~\ref{fig:maxincdrplot} shows the results for different values of $ L_i $ while keeping all other parameters at their default setting.
\begin{figure}
	\centering
	\includegraphics[width=0.7\linewidth]{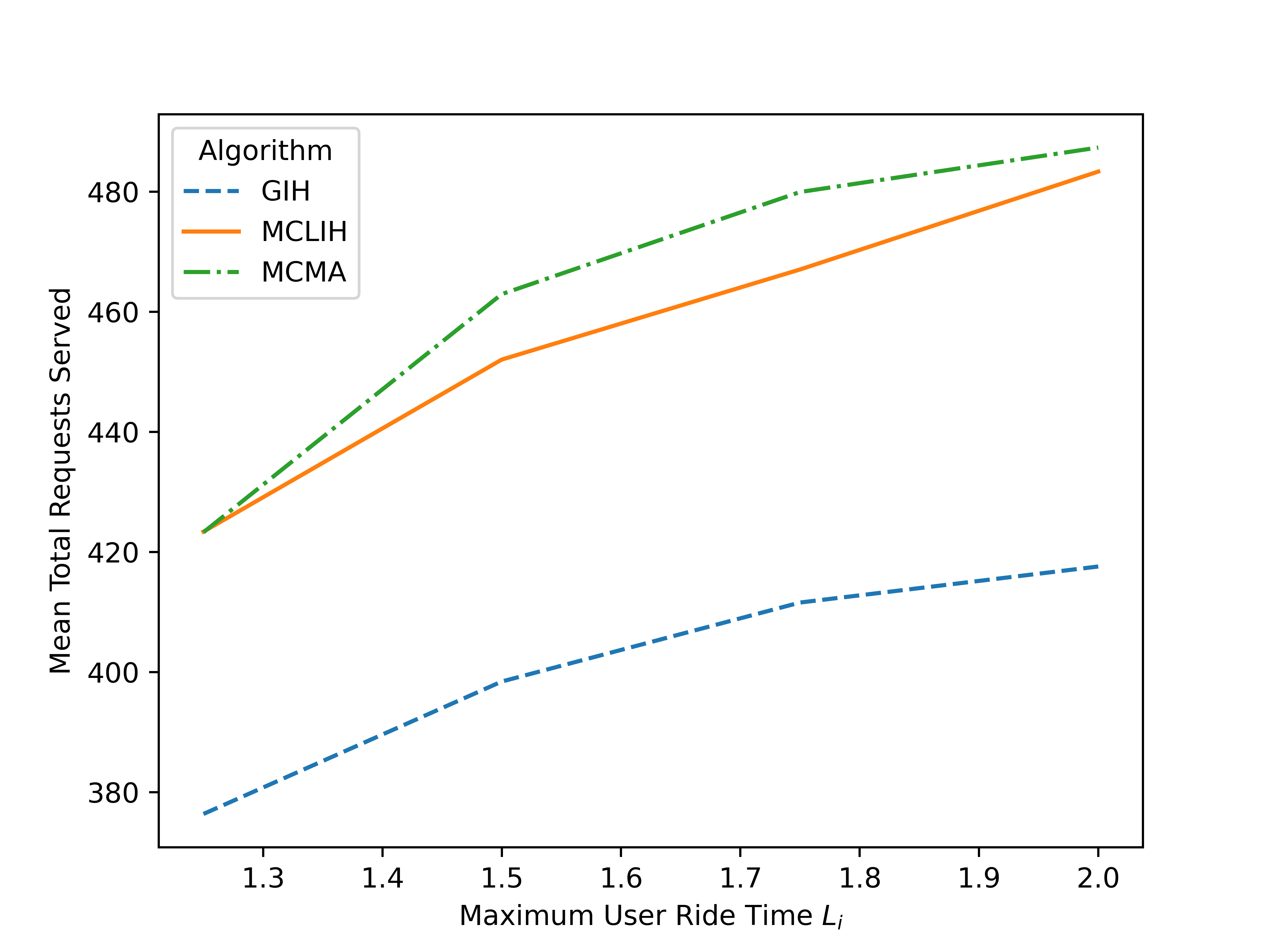}
	\caption{Mean number of total requests served for different maximum user ride times}
	\label{fig:maxincdrplot}
\end{figure}

First and most important, we can conclude from Figure~\ref{fig:maxincdrplot} that all three algorithms can profit from the increasing flexibility to a similar extent.
Both for MCLIH and MCMA the mean number of served requests increases by about $ 14\% $ when increasing the maximum user ride time from $ 1.25 $ to $ 2 $.
This reinforces our presumption from the preceding paragraph that the number of served requests is more sensitive to the vehicle capacity than to the maximum user ride time.

\paragraph{Evaluation of Different Mini-Cluster Parameters $ \rho $}
A different way of modifying the mini-clusters is to change the parameter $ \rho $ which determines how far apart two requests can be in order to belong to the same mini-cluster; compare Section \ref{sec:miniclustering}.
In order to have more variety in the mini-cluster sizes, we additionally increased the default vehicle size in these experiments to $ 5 $.
The results depicted in Figure \ref{fig:rhoPlot} show a somewhat similar trend as observed for the different vehicle capacities in Figure \ref{fig:increasingQ}.
\begin{figure}
	\centering
	\includegraphics[width=0.7\linewidth]{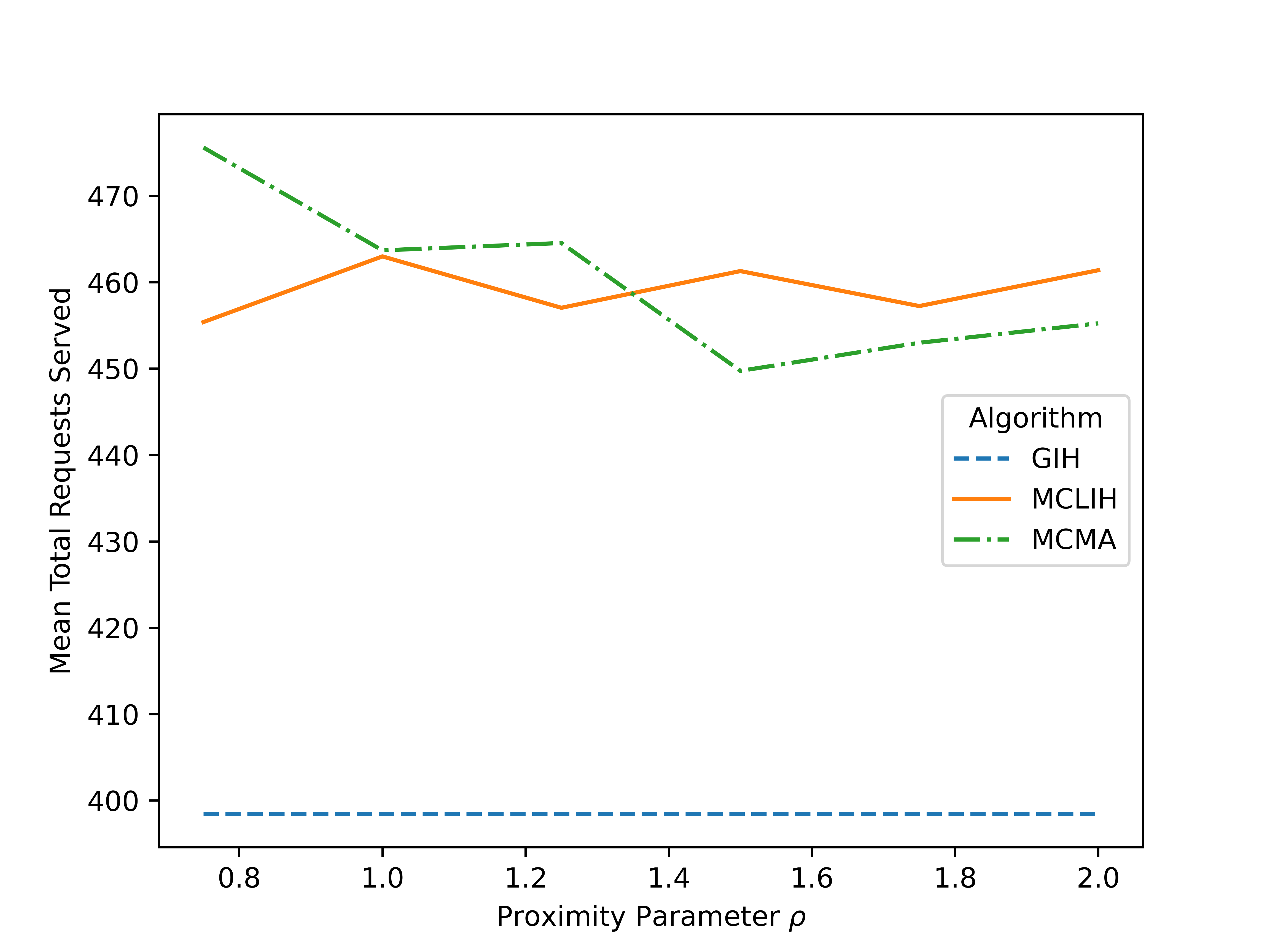}
	\caption{Mean number of total requests served for different values of $ \rho $ for $ Q=5 $}
	\label{fig:rhoPlot}
\end{figure}
For small values of $ \rho $ and thus more restrictive mini-clusters, MCMA performs slightly better whereas the relation between the performance of MCLIH and the choice of $ \rho $ does not seem to follow a clear trend.
Since GIH does not involve a mini-clustering phase, it is not affected by changes in $ \rho $ and we include its performance in the default setting with $ Q=5 $ for reference.

\paragraph{Summary}
Our sensitivity analyses revealed, that the overall behavior of MCMA and MCLIH appears to be stable with respect to evaluated parameters.
In fact, we observe that increasing the vehicle size can slightly improve the performance of MCLIH, whereas for MCMA it can result in ``overly optimistic'' mini-clusters which make the assignment of the return trips inefficient.
Similarly, MCMA performs better for restrictive choices of the mini-clustering parameter $ \rho $.
Increasing the maximum user ride time $ L_i $ (and similarly also the maximum time window $ W $) generally helps improving the schedules, but must be chosen moderately in a real setting.

\section{Conclusion}
\label{sec:conclusion}
In this paper, we introduced the DARPCF as a new extension of the standard DARP which has applications in customer transportation for customers with time flexibility.
It allows flexible scheduling of the outbound ride and requires that the corresponding inbound ride is scheduled within a predetermined time after the first ride.
We developed two algorithms, MCLIH and MCMA, which first create mini-clusters of similar outbound rides and then create vehicle schedules in different fashions.
For MCLIH, we adapted a route-first cluster-second method for capacitated vehicle routing problems.
The method solves a traveling salesman problem and splits the tour into sub-tours.
MCMA, on the other hand, uses a rolling horizon approach and creates vehicle routes by solving special bipartite matching problems between vehicles and jobs, so-called vehicle assignment matchings.

As the bases of our computational study, we generated transportation requests for a real-world rural primary care system and compared the new algorithms to the regular DARP setting when using a greedy insertion heuristic.
Within our experiments, an increase of roughly $ 16\% $ for the number of served requests could be obtained.
An integration of the DARPCF into the simulation tool SiM-Care~\cite{Comis19} will help to estimate the influence of flexible scheduling in more detail.

A limitation of the presented flexible dial-a-ride system is the required flexibility of the customers.
As such, it may turn out that the mentioned $ 16\% $ increase in rides might not suffice to convince operators to change the way requests and appointments are handled.
Nonetheless, we have seen that the advantages of a flexible scheduling depend on multiple factors, e.g.\ the number of chronic patients and the maximum user ride time.
Hence, the added value of such a system can be substantially higher in other settings.
Moreover, we see room for several improvements of presented algorithms, e.g.\ through more advanced mini-clustering techniques like a neighborhood search.
Next to these improvements, future work should investigate algorithms for the DARPCF that can exploit the information that each outbound ride is followed by an inbound ride.
Finally, we would like to investigate how historical data or probabilities for short-notice requests or appointment durations can be used to further improve the vehicle schedules.

\section*{Acknowledgements}
	This work was supported by the Freigeist-Fellowship of the Volkswagen Stiftung; the German research council (DFG) Research Training Group 2236 UnRAVeL; and the German Federal Ministry of Education and Research (grant no. 05M16PAA) within the project ``HealthFaCT - Health: Facility Location, Covering and Transport''.

\bibliography{DARP_article}

\appendix
\section*{Appendix: Pseudocode MCLIH}

\begin{algorithm}[h]
	\caption{Mini-Cluster Linking Insertion Heuristic (MCLIH)
	}
	\label{alg:MCLIH}

	\hspace*{\algorithmicindent} \begin{minipage}{\dimexpr\textwidth-\algorithmicindent}
		\textbf{Input:} Transportation requests $ R=R_c \cup \overline{R_c} \cup R_w \cup \overline{R_w} $ with corresponding DARP road graph $ G=(N,A) $,
		vehicle fleet size $ m $,
		vehicle capacity $ Q $,
		maximum route length $ L $,
		maximum user ride time $ L_i $ for each request $ r_i $,
		mini-clustering parameter $ \rho $.
	\end{minipage}
	\begin{algorithmic}[1]
		\STATE Compute set of mini-clusters $ \mathcal{V} $ with Algorithm \ref{alg:greedy-partition}
		\FOR{$ \mathcal{M} \in \mathcal{V} $}
			\STATE Compute intra-cluster route $ \mathcal{S} $ on $ \mathcal{M} $ by solving formulation \eqref{openTSP1}--\eqref{eq:lastDARPmilp}
		\ENDFOR
		\STATE Construct complete directed auxiliary graph  $ H = (\mathcal{V}, \mathcal{A}, c) $ where $ c(\mathcal{M}_i,\mathcal{M}_j) $ reflects distance between last stop of $ \mathcal{S}_i $ and first stop of $ \mathcal{S}_j $
		\STATE Compute TSP tour $ T $ on $ H $
		\STATE Construct auxiliary graph $ H'=(\mathcal{V},E') $ modeling all feasible vehicle routes obtained by splitting $ T $; see \cite{PRINS20041985,Labadie2016}
		\STATE Compute shortest path on $ H' $, obtaining vehicle schedules $ \mathcal{R}_1, \mathcal{R}_2, \dots, \mathcal{R}_m $
		\FOR{$ r\in R_c $}
		\STATE Set appointment time after arrival of $ r $
		\STATE Construct time windows for corresponding inbound request $ \overline{r} \in \overline{R_c} $
		\STATE Insert $ \overline{r} $ into $ \mathcal{R}_1, \mathcal{R}_2, \dots, \mathcal{R}_m $ using GIH (if it fails: remove and re-insert $ r $ and $ \overline{r} $ with different time windows)
		\ENDFOR
		\STATE Notify chronic patients about appointments, start handling walk-in rides on request
		\FOR{$ (r,\overline{r})\in (R_w \cup \overline{R_w}) $}
		\STATE Insert $ r $ and $ \overline{r} $ into $ \mathcal{R}_1, \mathcal{R}_2, \dots, \mathcal{R}_m $ by GIH in an online fashion
		\IF{Insertion successful}
		\STATE \textbf{continue}
		\ELSE
		\STATE Reject $ (r,\overline{r})$
		\ENDIF
		\ENDFOR
		\RETURN $ \mathcal{R}_1, \mathcal{R}_2, \dots, \mathcal{R}_m $
	\end{algorithmic}
\end{algorithm}

\end{document}